\documentclass[11pt, reqno]{amsart}

\usepackage{amsthm,amssymb,amstext,amscd,amsfonts,amsbsy,amsrefs,amsxtra,latexsym,amsmath,xcolor,mathrsfs,fancybox,upgreek, soul,url}
\usepackage[top=0.9in, bottom=0.9in, left=1.25in, right=1.25in]{geometry}
\usepackage[english]{babel}
\usepackage[all,cmtip]{xy}
\usepackage[latin1]{inputenc}
\usepackage{cancel}
\usepackage{comment}
\usepackage{mdframed}
\usepackage{mathtools}
\usepackage{enumerate}
\usepackage{thmtools}
\usepackage{thm-restate}
\usepackage{enumitem}
\usepackage{etoolbox}
\usepackage{todonotes}
\usepackage{footmisc}
\usepackage{bigints}

\usepackage[colorlinks=true,linkcolor=black,citecolor=black,filecolor=black]{hyperref}

\newtheorem{theorem}{Theorem}
\newtheorem*{question}{Question}
\newtheorem{lemma}[theorem]{Lemma}

\theoremstyle{definition}
\newtheorem*{definition}{Definition}

\theoremstyle{remark}

\newtheorem*{remarks}{Remarks}

\numberwithin{theorem}{section}
\numberwithin{proposition}{section}
\numberwithin{lemma}{section}
\numberwithin{corollary}{section}
\numberwithin{equation}{section}
\numberwithin{conjecture}{section}
\numberwithin{example}{section}

\setlist[enumerate,1]{before=}
\AfterEndEnvironment{enumerate}{}

\newcommand{\Z}{\mathbb{Z}}
\newcommand{\R}{\mathbb{R}}
\newcommand{\C}{\mathbb{C}}

\newcommand{\SL}{{\text {\rm SL}}}

\newcommand{\sgn}{\operatorname{sgn}}

\newcommand{\im}{\textnormal{Im}}

\def\H{\mathbb{H}}
\renewcommand{\pmod}[1]{\  \,  \left( \mathrm{mod} \,  #1 \right)}

\renewcommand{\pmod}[1]{\  \,  \left( \mathrm{mod} \,  #1 \right)}

\begin{document}

\title{Higher depth mock theta functions and $q$-hypergeometric series}

\author{Joshua Males}
\address{450 Machray Hall, Department of Mathematics, University of Manitoba, Winnipeg,
	Canada}
\email{joshua.males@umanitoba.ca}
\author{Andreas Mono}
\address{Department of Mathematics and Computer Science, Division of Mathematics, University of Cologne, Weyertal 86-90, 50931 Cologne, Germany}
\email{amono@math.uni-koeln.de}
\author{Larry Rolen}
\address{Department of Mathematics, 1420 Stevenson Center, Vanderbilt University, Nashville, TN 37240}
\email{larry.rolen@vanderbilt.edu}

\maketitle

\begin{abstract}
In the theory of harmonic Maa{\ss} forms and mock modular forms, mock theta functions are distinguished examples which arose from $q$-hypergeometric examples of Ramanujan. Recently, there has been a body of work on higher depth mock modular forms. Here, we introduce distinguished examples of these forms which we call higher depth mock theta functions and develop $q$-hypergeometric expressions for them.
We provide three examples of mock theta functions of depth two, each arising by multiplying a  classical mock theta function with a certain specialization of a universal mock theta function. In addition, we give their modular completions, and relate each to a $q$-hypergeometric series.
\end{abstract}

\section{Introduction}

The study of mock theta functions goes back to Ramanujan, who gave the first examples in his enigmatic last letter to Hardy one hundred years ago. In $2002$, Zwegers \cite{zwegers2008mock} achieved a major breakthrough by recognizing Ramanujan's classical mock theta functions as holomorphic parts of so-called harmonic Maa{\ss} forms, which in turn enabled him to provide non-holomorphic completions of Ramanujan's functions to modular objects. Since then, there has been an enormous amount of interest in, and new results related to, mock theta functions and harmonic Maa{\ss} forms\footnote{A definition can be found in \cite[Chapter 4]{thebook} for instance.}. For example, the landmark paper of Bringmann and Ono \cite{bron10} showed a deep connection between the ranks of partitions, mock theta functions, and harmonic Maa{\ss} forms. To describe this, we define 
\begin{align*}
\mathcal{R}(\alpha, \beta;q) \coloneqq \sum_{n \geq 0} \frac{(\alpha \beta)^n q^{n^2}}{(\alpha q;q)_n (\beta q;q)_n},
\end{align*}
where 
\begin{align*}
(a)_n \coloneqq (a;q)_n &\coloneqq \prod_{j=0}^{n-1} (1-aq^j)
\end{align*}
is the usual $q$-Pochhammer symbol and $q \coloneqq e^{2 \pi i \tau}$ with $\tau \in \H$ throughout. The function $\mathcal{R}$ was studied by Folsom \cite{folsom} and is the three-variable generalization of the partition rank generating function $\mathcal{R}\left(w,w^{-1};q\right)$. In particular, $\mathcal{R}$ is in essence a universal mock theta function in the sense of Gordon and McIntosh \cite{gordmc}. That is, many of Ramanujan's original mock theta functions may be written as specializations of $\mathcal{R}$, up to the addition of a modular form - see \cite[p. 490]{folsom}, and also \cite[Theorem 3.1]{brforh}.

Returning to the results by Bringmann and Ono \cite{bron10}, they showed that $\mathcal{R}\left(\xi,\xi^{-1};q\right)$ evaluated at an odd order root of unity $\xi \neq 1$ is a mock modular form of weight $\frac{1}{2}$ with a certain shadow. The knowledge of the modularity behaviour then allows one to obtain deep arithmetical information on ranks of partitions, including on their asymptotics and exact formulae \cite{Bri,bron06}, as well as congruences that are satisfied, in analogy to the famed Ramanujan congruences. Afterwards, Zagier \cite{zagier} provided a new proof of Bringmann and Ono's result on the rank generating function in an expository paper, and his proof applies to an even order root of unity $\xi \neq 1$ as well.


Consequently, the investigation of $q$-hypergeometric series became a leitmotif in the area of combinatorics as well as in the area of mock modular forms. In many cases, the focus of active research in both of these fields originates in the investigation of some peculiar explicit examples, which shed the first light on a new phenomenon or object. To name one such example, Lovejoy and Osburn \cite{loos12, loos16} offered a new perspective on some of the classical mock theta functions. In short, they discovered four examples of double sum $q$-series which are also mock theta functions. In other words, their work can be regarded as the observation that some (if not all) mock theta functions are double sum $q$-series, which happen to collapse to single sums. Two such functions are $\mathcal{M}_{10}$ and $\mathcal{M}_{17}$ in \cite{loos16}, explicitly given by
\begin{align} \label{eq: loos-ex}
\sum_{n \geq 1} \sum_{n \geq j \geq 1} \frac{ (-1)^j (-1)_{2n} q^{j^2+j+n}}{(q^2;q^2)_{n-j} (q^2;q^2)_{j-1} (1-q^{4j-2})}, \quad \sum_{n \geq 1} \sum_{n \geq j \geq 1} \frac{ (-1)^j (-1)_{2n} q^{j^2-j+n}}{(q^2;q^2)_{n-j} (q^2;q^2)_{j-1} (1-q^{4j-2})}.
\end{align}
Both are expressible in terms of Zwegers' $\mu$-function (defined below) and ratios of classical theta series, c.f. \cite[Theorem 1.6, 1.7]{loos16}, and consequently are mock theta functions.

A second example of a new phenomenon or object is the notion of higher depth mock modular forms, which arose parallel to the work of Lovejoy and Osburn. Roughly speaking, classical mock modular forms can be viewed as preimages of (weakly) holomorphic modular forms under the differential operator 
\begin{align*}
\xi_k \coloneqq 2i\im(\tau)^k\overline{\frac{\partial}{\partial \overline{\tau}}},
\end{align*}
introduced by Bruinier, Funke \cite{bruinierfunke2004}, and which is surjective on the space of weight $k$ harmonic Maa{\ss} forms. Since the kernel of $\xi_k$ contains precisely the holomorphic functions, this idea generalizes directly to mixed mock modular forms, which were first defined in \cite{damuza} (see \cite[Section 13.2]{thebook} for a definition as well). Consequently, weakly holomorphic modular forms can be regarded as mock modular forms of depth zero. Now, one can define mixed mock modular forms of depth $d$ inductively as preimages of mixed mock modular forms of depth $d-1$ under $\xi_k$ (see Section \ref{Section: prelim higher depth} for a precise definition). In other words, resembling the extension of holomorphic modular forms to holomorphic quasimodular forms, higher depth mock modular forms extend the scope of admissible images under $\xi_k$ in a natural fashion. Recently, such forms were connected to black holes by Alexandrov and Pioline \cite{alpi18}, to the Gromov-Witten theory of elliptic orbifolds \cite{brkaro}, and to indefinite theta functions on arbitrary lattices of signature $(r,n-r)$ - see \cite{abmp} for the $r=2$ case and \cite{Naz} for general $r$, each of which are generalizations of Zwegers' groundbreaking thesis \cite{zwegers2008mock} where $r=1$. There are also further applications after relaxing to the notion of higher depth quantum modular forms \cite{BriKaMi,BriKaMi2,males}. 

The $q$-hypergeometric structure of examples of mock theta functions is also crucial to applications in geometry and topology. For instance, Nahm's Conjecture asserts that certain $q$-hypergeometric series are modular if and only if some associated elements of a $K$-theoretic group (the Bloch group) is torsion. Zagier brilliantly proved this in the case of rank $1$ \cite{ZagierNahm}. While it turned out to not be true in higher rank cases (as shown by Vlasenko and Zwegers \cite{VZ}), Calegari, Garoufalidis, and Zagier \cite{CGZ} later showed that one direction of it was true in general. Their proof was closely related to objects in knot theory. Indeed, there are procedures whereby knot diagrams produce $q$-hypergeometric series of a similar shape as Ramanujan's mock theta functions (see, e.g., \cite{Gar}). The proofs of the above cases of Nahm's Conjecture relied on the same sort of asymptotic analysis near roots of unity that Ramanujan employed to discover the mock theta functions. It could have been in families like the one Zagier studeid in \cite{ZagierNahm} that some examples were mock modular, and not just modular, but they didn't happen to turn up there. However, such series in more general contexts may play a role in more exotic modular constructions, motivating studies such as that in this paper. The asymptotics of these knot $q$-series are closely connected to quantum modular forms and the important Volume Conjecture which now seems to be heavily tied to modularity-type properties \cite{ZagQuantum}.

To our knowledge, all higher depth mock modular forms beyond the original sorts considered in Zwegers' thesis are constructed as indefinite theta series, but not via other means which were historically important in the development of the original mock modular forms.  In light of these observations from combinatorics and topology, it is natural to ask the following.
\begin{question}
Are there interesting $q$-hypergeometric depth $\geq2$ mock modular forms?
\end{question}
This paper aims to answer this question and to propose a new structure inductively extending Ramanujan's original mock theta functions to a set of spaces of distinguished higher depth mock modular forms.

Classical mock theta functions are special mock modular forms of depth one, which are in addition required to be of weight $\frac{1}{2}$ or $\frac{3}{2}$, and whose image under the $\xi$-operator is a linear combination of unary theta functions. In this paper, we focus on a new class of objects closely related to higher depth mixed mock modular forms, which we call higher depth mock theta functions. We define them explicitly in Section \ref{Section: prelim higher depth}, and we construct the first examples of depth two mock theta functions. Being more precise, we focus on three of Ramanujan's order three mock theta functions throughout, given by
\begin{align*}
\nu(q) \coloneqq \sum_{n \geq 0} \frac{q^{n(n+1)}}{(-q;q^2)_{n+1}},\qquad
\phi(q) \coloneqq \sum_{n \geq 0} \frac{q^{n^2}}{(-q^2;q^2)_n},\qquad
\rho(q) \coloneqq \sum_{n \geq 0} \frac{q^{n^2}}{(q;q^2)_n}.
\end{align*}
Note that we also cover the cases of the order three mock theta functions $\omega$ and $\psi$ implicitly\footnote{c.f. \cite[Appendix A.2]{thebook}, for instance} by virtue of their simple relationships to those that appear here. Next, we multiply each of the three mock theta functions by a certain specialization of $\mathcal{R}$.

Our main result shows that these first examples of higher depth mock theta functions arise as double-sum $q$-hypergeometric functions. Throughout we let $\zeta \coloneqq e^{2 \pi i z}$ with $z \in \C$, and let
\begin{align*}
\left[{m}\atop{n}\right]_{q} \coloneqq \frac{(q;q)_m}{(q;q)_{m-n}(q;q)_n}
\end{align*}
be the $q$-binomial coefficient. We define
\begin{align*}
	& f_1(z,\tau) \coloneqq (1+\nu(q))\left( 1+\frac{\zeta}{(1-\zeta)(1+q)}\mathcal{R}(\zeta,-q;q^2)  \right),\\
	&	f_2(z,\tau) \coloneqq \phi(q) \left(1 + \frac{\zeta}{(1-\zeta)(1+q^2)} \mathcal{R}(\zeta,-q^2;q^2)\right),\\
	&	f_3(z,\tau) \coloneqq \rho(q) \left(1+ \frac{\zeta}{(1-\zeta)(1-q)} \mathcal{R}(\zeta,q;q^2)\right),
\end{align*}
and have the following result.
\begin{theorem}\label{Theorem: main}
	Let $\zeta$ be a root of unity. Then the functions $f_j$ for $j=1,2,3$ are each mock theta functions of depth two. Furthermore, we have the following representations as double-sum $q$-series.
	\begin{enumerate}[wide, labelwidth=!, labelindent=0pt]
		\item The function $f_1$ can be written as
			\begin{align*}
			f_1(z,\tau) = (1+q^{-1}) \sum_{m,n \geq 0} (-1)^n q^{2n^2}\zeta^{n+m} \left[{m+n}\atop{m}\right]_{q^2} \frac{(-q^{2n}/\zeta;q^2)_m }{(1+q^{2n-1})(-q;q^2)_{m+2n}}.
			\end{align*}
		\item The function $f_2$ can be written as
			\begin{align*}
			f_2(z,\tau) = 2\sum_{m,n \geq 0} (-1)^n q^{2n^2+n}\zeta^{n+m} \left[{m+n}\atop{m}\right]_{q^2} \frac{(-q^{2n+1}/\zeta;q^2)_m}{(1+q^{2n})(-q^2;q^2)_{m+2n}}.
			\end{align*}
		\item The function $f_3$ can be written as
			\begin{align*}
			f_3(z,\tau) = \sum_{m,n \geq 0} (-1)^n q^{2n^2+n}\zeta^{n+m} \left[{m+n}\atop{m}\right]_{q^2} \frac{(-q^{2n+1}/\zeta;q^2)_m(1-q^{-1})}{(1-q^{2n-1})(q;q^2)_{m+2n}}.
			\end{align*}
	\end{enumerate}
	
\end{theorem}

\begin{remarks}
\begin{enumerate}[wide, labelwidth=!, labelindent=0pt]
\item Each of the three cases suggest choosing $\zeta = q^{\ell}$ for some $\ell \geq 1$ to achieve more cancellation.
\item We emphasize that the $q$-series on the right-hand side of Theorem \ref{Theorem: main} may be viewed combinatorially as counting certain families of partitions, in analogy with the depth one case. For succinctness, we do not provide explicit details here.
\item We highlight that our results and the examples \eqref{eq: loos-ex} of Lovejoy, Osburn have a similar shape. Further evidence of this connection is found in the fact that Zwegers' $\mu$-function essentially provides the completion of both the double-sum $q$-series in equation \eqref{eq: loos-ex} and in Theorem \ref{Theorem: main}. 
\end{enumerate}
\end{remarks}

\begin{theorem}\label{Theorem main: completions}
	Each of the functions $f_j$  has a natural  modular completion (see Section \ref{sec:modcom}).
\end{theorem}

To our knowledge, these are the first examples of depth two mock theta functions written as $q$-hypergeometric functions.  Invoking different product formulae for single-sum $q$-hypergeometric functions may yield a broader set of higher depth mock theta functions, and a natural question we pose is the following.

\begin{question}
	Are there further examples of higher depth mock theta functions? Are there higher depth mock theta functions that do not arise as products of those of lower depth?
\end{question}

In analogy to the depth one case, it is clear that there is a combinatorial interpretation of the $q$-hypergeometric series described in Theorem \ref{Theorem: main}. Furthermore, the modularity properties yield asymptotics for the coefficients using standard techniques, and we offer the following general question.

\begin{question}
	What are the applications of higher depth mock theta functions?
\end{question}

\section*{Acknowledgments}
We would like to thank Jeremy Lovejoy for insightful discussions on partition-theoretic aspects of this paper, and for valuable comments to Section \ref{Section: prelim higher depth}. In addition, we would like to thank Kathrin Bringmann for useful comments on an earlier version of this paper. We also thank the referee for many helpful suggestions, and Matthias Storzer for pointing out a small error in our description of Nahm's conjecture.

\section{Preliminaries}
In this section we collect some preliminary results and definitions pertinent to the rest of the paper.
\subsection{$q$-hypergeometric series}
We utilize the $q$-hypergeometric series
\begin{align*}
{}_r\phi_{s}\left(a_1,\ldots,a_r;b_1,\ldots,b_s;q;z\right) &\coloneqq \sum_{n \geq 0} \frac{(a_1;q)_n \cdots (a_r;q)_n}{(b_1;q)_n \cdots (b_s;q)_n} \frac{z^n}{(q;q)_n} \left((-1)^n q^{\frac{n(n-1)}{2}}\right)^{s-r+1},
\end{align*}
which enables us to state a central formula in our work.
\begin{lemma}[\protect{\cite[equation (2.10)]{sriva}}]
We have
\begin{align} \label{eqn: product}
{}_1\phi_1(\lambda;\mu;q;-z)&{}_2\phi_1(\lambda,0;\mu;q;\zeta) \notag
\\
&=
\sum_{m,n\geq0}q^{n(n-1)}\frac{(\lambda;q)_{m+n}(-q^nz/\zeta;q)_m(\mu/\lambda;q)_n}{(\mu;q)_{m+2n}(\mu;q)_n}\cdot\frac{\zeta^m}{(q;q)_m}\cdot\frac{(-\lambda z\zeta)^n}{(q;q)_n}.
\end{align}
\end{lemma}


In addition, we use Fine's function \cite{fine}
\begin{align*}
F(a,b;t;q) \coloneqq \sum_{n \geq 0} \frac{(aq;q)_n}{(bq;q)_n} t^n,
\end{align*}
which has the following transformation properties (among others).
\begin{lemma}[\protect{\cite[eq. (4.4), (6.3), (12.3)]{fine}}]
It holds that
\begin{align}
F(a,b;t;q) &=  \frac{b}{b-at} + \frac{(b-a)t}{(1-bq)(b-at)} F(a,bq;t;q), \label{eqn: fine 4.4}\\
F(a,b;t;q) &=  \frac{1-b}{1-t}F\left(\frac{at}{b},t;b;q\right). \label{eqn: fine 6.3} \\
(1-t)F(0,b;t;q) &= \sum_{n \geq 0} \frac{(bt)^n q^{n^2}}{(bq;q)_n (tq;q)_n}. \label{eqn: fine 12.3}
\end{align}
\end{lemma}


\subsection{Higher depth mock modular forms}\label{Section: prelim higher depth}
Here we recall the precise definition of higher depth mock modular forms, following talks of Zagier$-$Zwegers, and work of Westerholt-Raum \cite{raum}. Let $\Gamma \subseteq \SL_2(\Z)$ be a congruence subgroup, and let $M_k(\Gamma)$ be the space of holomorphic modular forms of weight $k$ on $\Gamma$. We say that a function $f \colon \H \to \C$ transforms like a modular form of weight $k$ on $\Gamma$, if for all $\gamma \in \Gamma$ and all $\tau \in \H$ we have 
\begin{align*}
f(\tau) = \begin{cases}
(c\tau+d)^{-k} f(\gamma\tau) & \text{if} \ k \in \Z, \\
\left(\frac{c}{d}\right)\varepsilon_d^{2k}(c\tau+d)^{-k} f(\gamma\tau) & \text{if} \ k \in \frac{1}{2}+\Z,
\end{cases}
\end{align*}
where $\left(\frac{c}{d}\right)$ denotes the extended Legendre symbol, and
\begin{align*}
\varepsilon_d \coloneqq \begin{cases}
1 & \text{if} \ d \equiv 1 \pmod{4}, \\
i & \text{if} \ d \equiv 3 \pmod{4}, \\
\end{cases}
\end{align*}
for odd integers $d$.

\begin{definition}
A \emph{modular completion} of a function $f \colon \H \to \C$ on $\Gamma$ is a function $g \colon \H \to \C$, such that $f+g$ transforms like a modular form of some weight on $\Gamma$.
\end{definition}
Note that a modular completion is not unique. Indeed, one may add a modular form of suitable weight (or a more general automorphic function) to such a completion. Thus, we emphasize that a \emph{natural} modular completion $g$ should provide new insight on the obstruction towards modularity of the initial function $f$. 

\begin{definition}[\protect{\cite[\S 4.2]{raum}}]
Let $\mathcal{M}_k^0(\Gamma) \coloneqq M_k(\Gamma)$. For $d>0$, the space $\mathcal{M}_k^d (\Gamma)$ of \textit{mock modular forms of depth $d$ and weight $k$ on $\Gamma$} is the space of real-analytic functions on $\H$ 
\begin{enumerate}[wide, labelwidth=!, labelindent=0pt]
\item that admit a modular completion of weight $k$ on $\Gamma$,
\item whose images under $\xi_k$ are contained in the space\footnote{The symbol $\otimes$ refers to the usual tensor product of vector spaces.}
\begin{align*}
\sum_{\ell} \overline{M_\ell(\Gamma)} \otimes \mathcal{M}_{k-\ell}^{d-1}(\Gamma),
\end{align*}
\item and that are of at most linear exponential growth towards the cusps of $\Gamma$.
\end{enumerate} 
\end{definition}

We observe that mock modular forms of depth one are precisely the mixed mock modular forms. 
In a similar fashion, we now define higher depth mock theta functions. 
\begin{definition}
Let $\Theta_{\frac{1}{2}}(\Gamma)$, $\Theta_{\frac{3}{2}}(\Gamma)$ be the space of unary theta functions of weight $\frac{1}{2}$ or $\frac{3}{2}$ on $\Gamma$ respectively. In addition, let $\mathbb{M}_{k}^{0}(\Gamma) \coloneqq M_k(\Gamma)$. For $d > 0$, the space $\mathbb{M}_k^d (\Gamma)$ of \textit{mock theta functions of depth $d$ and weight $k$ on $\Gamma$} is the space of real-analytic functions on $\H$ that
\begin{enumerate}[wide, labelwidth=!, labelindent=0pt]
\item that admit a modular completion of weight $k$ on $\Gamma$,
\item have images under $\xi_k$ that are contained in the space
\begin{align*}
\left(\Theta_{\frac{1}{2}}(\Gamma) \otimes \mathbb{M}_{k-\frac{1}{2}}^{d-1}(\Gamma)\right) + \left(\Theta_{\frac{3}{2}}(\Gamma) \otimes \mathbb{M}_{k-\frac{3}{2}}^{d-1}(\Gamma)\right),
\end{align*}
\item are of at most linear exponential growth towards the cusps of $\Gamma$.
\end{enumerate}
\end{definition}

Following Zagier \cite{zagier}, we observe once more that mock theta functions of depth one are precisely the classical mock theta functions multiplied by modular forms. 

\subsection{Modular completions} \label{sec:modcom}
In this section we collect the modular completions of $\nu, \phi, \rho$, and $\mathcal{R}$. Suppose that $u,v \in \C \setminus \left(\Z\tau+\Z\right)$. Zwegers \cite{zwegers2008mock} defined his $\mu$-function by
\begin{align*}
\mu\left(\mathrm{e}^{2\pi i u},\mathrm{e}^{2\pi i v};\tau\right) \coloneqq \frac{\mathrm{e}^{i\pi u}}{-iq^{\frac{1}{8}} \mathrm{e}^{-i \pi v} (q;q)_{\infty}\left(\mathrm{e}^{2\pi i v};q\right)_{\infty}\left(\mathrm{e}^{-2\pi i v}q;q\right)_{\infty}} \sum_{n \in \Z} \frac{(-1)^n q^{\frac{n(n+1)}{2}}\mathrm{e}^{2\pi i n v}}{1-\mathrm{e}^{2\pi i u}q^n},
\end{align*}
where we used the Jacobi triple product identity. To describe a natural modular completion of $\mu$, we recall the error function
\begin{align*}
E(z) \coloneqq 2 \int_{0}^{z} e^{-\pi t^2} dt
\end{align*}
for $z \in \R$, along with
\begin{align*}
R(u;\tau) \coloneqq \sum_{n \in \frac{1}{2}+\Z} \left[\sgn(n) - E\left(\left(n+\frac{\im(u)}{\im(\tau)}\right)\sqrt{2\im(\tau)}\right)\right] (-1)^{n-\frac{1}{2}} e^{-\pi i n^2\tau - 2 \pi i n u}.
\end{align*}

Following Choi \cite{choi}, we additionally consider
\begin{align*}
\mathcal{U}(\alpha,\beta;q) &\coloneqq \sum_{n \geq 1} \left(\alpha^{-1};q\right)_n \left(\beta^{-1};q\right)_n q^n, \\
\mathcal{M}(u,v,\tau) &\coloneqq iq^{\frac{1}{8}} \left(1-\mathrm{e}^{2\pi i u}\right) \mathrm{e}^{\pi i(v-u)} \left(\mathrm{e}^{2\pi i (\tau-u)};q\right)_{\infty} \left(\mathrm{e}^{-2\pi i v};q\right)_{\infty} \mu\left(\mathrm{e}^{2\pi i u},\mathrm{e}^{2\pi i v};\tau\right).
\end{align*}

\begin{lemma}\label{Lem: R completion}
The function $\mathcal{R}(e^{2 \pi i u},e^{2\pi i v};q)$ admits a completion given by
\begin{align*}
&\mathcal{C}(e^{2 \pi i u},e^{2 \pi i v};q) \\
&\coloneqq -\frac{q^{\frac{1}{8}}}{2} \left(1-\mathrm{e}^{2\pi i u}\right) \mathrm{e}^{\pi i(v-u)} \left(\mathrm{e}^{2\pi i (\tau-u)};q\right)_{\infty} \left(\mathrm{e}^{-2\pi i v};q\right)_{\infty}R(u-v;\tau) + \mathcal{U}\left(\mathrm{e}^{2\pi i u},\mathrm{e}^{2\pi i v};q\right).
\end{align*}
\end{lemma}

\begin{proof}	
Ramanujan's identity \cite[p. 67, entry 3.4.7]{lostnb2} can be rewritten as
\begin{align*}
\mathcal{M}(u,v,\tau) = \mathcal{R}\left(\mathrm{e}^{2\pi i u},\mathrm{e}^{2\pi i v};q\right) + \mathcal{U}\left(\mathrm{e}^{2\pi i u},\mathrm{e}^{2\pi i v};q\right),
\end{align*}
compare \cite[Theorem 4]{choi} as well. By Zwegers' thesis \cite[Theorem 1.11]{zwegers2008mock}, the modular completion of $\mathcal{M}(u,v;\tau)$ is given by
\begin{align*}
C(u,v;\tau) \coloneqq -\frac{q^{\frac{1}{8}}}{2} \left(1-\mathrm{e}^{2\pi i u}\right) \mathrm{e}^{\pi i(v-u)} \left(\mathrm{e}^{2\pi i (\tau-u)};q\right)_{\infty} \left(\mathrm{e}^{-2\pi i v};q\right)_{\infty}  R(u-v;\tau).
\end{align*}
This in turn produces the modular completion of $\mathcal{R}$ as $C(u,v;\tau) + \mathcal{U}(\mathrm{e}^{2\pi i u},\mathrm{e}^{2\pi i v};q) $.
\end{proof} 

The modular completions of Ramanujan's mock theta functions are known by Zwegers' thesis \cite{zwegers2008mock}, and their representations in terms of $\mu$ (see \cite[Appendix A.2.]{thebook}, for example). For convenience, we recall the modular completions here without proof. 
\begin{lemma}\label{Lem: mock theta completion}
We have the following modular completions.
\begin{enumerate}[wide, labelwidth=!, labelindent=0pt]
\item The mock theta function $\nu(q)$ admits a completion given by 
\begin{align*}
-q^{-\frac{1}{2}} R(2\tau;12\tau).
\end{align*}
		
\item The mock theta function $\phi(q)$ admits a completion given by
\begin{align*}
-e^{\frac{\pi i}{8}}q^{-\frac{1}{8}} R\left(-\tau;3\tau+\frac{1}{2}\right).
\end{align*}
		
\item The mock theta function $\rho(q)$ admits a completion given by
\begin{align*}
-\frac{1}{2}q^{-\frac{3}{4}} R(\tau;6\tau).
\end{align*}
\end{enumerate}
\end{lemma}

\section{{\bf Proof of Theorems  \ref{Theorem: main} and \ref{Theorem main: completions}}}\label{Section: theorem proofs}
We prepare the application of \eqref{eqn: product} with a lemma. First note that it is easy to show that
\begin{align}
1+\nu(q) &= {}_1\phi_1(q^2;-q;q^2;-1), \label{eqn: nu as 1phi1}\\
\phi(q) &= {}_1\phi_1(q^2;-q^2;q^2;-q), \label{eqn: phi as 1phi1}\\
\rho(q) &= {}_1\phi_1(q^2;q;q^2;-q). \label{eqn: rho as 1phi1}
\end{align}
Then we prove the following.
\begin{lemma}\label{Lem: mocks to R}
	We have the identities
	\begin{align*}
	& {}_2\phi_1(q^2,0;-q;q^2;\zeta) = 1 + \frac{\zeta}{(1-\zeta)(1+q)}\mathcal{R}(\zeta,-q;q^2),\\
	& {}_2\phi_1(q^2,0;-q^2;q^2;\zeta) = 1 + \frac{\zeta}{(1-\zeta)(1+q^2)} \mathcal{R}(\zeta,-q^2;q^2),\\
	& {}_2\phi_1(q^2,0;q;q^2;\zeta) = 1+ \frac{\zeta}{(1-\zeta)(1-q)} \mathcal{R}(\zeta,q;q^2).
	\end{align*}
\end{lemma}

\begin{proof}
To verify the first equation, we first rewrite the left hand side in terms of Fine's function $F$, namely
\begin{align*}
{}_2\phi_1(q^2,0;-q;q^2;\zeta) = F(0,-q^{-1};\zeta;q^2).
\end{align*}
Then by \eqref{eqn: fine 4.4} we find that
\begin{align*}
F(0,-q^{-1}; \zeta;q^2) = 1+\frac{\zeta}{1+q^2} F(0,-q;\zeta; q^2).
\end{align*}
Next, by \eqref{eqn: fine 6.3}, we obtain
\begin{align*}
F(0,\zeta;-q;q^2) = \frac{1-\zeta}{1+q}F(0,-q;\zeta;q^2).
\end{align*}
Thus,
\begin{align*}
F(0,-1;\zeta;q^2) = 1+\frac{\zeta}{1-\zeta} F(0,\zeta;-q;q^2).
\end{align*}
Using \eqref{eqn: fine 12.3}, we arrive at
\begin{align*}
F(0,\zeta;-q;q^2) = \frac{1}{1+q} \left( \sum_{n \geq 0} \frac{(-q\zeta)^n q^{2n^2}}{(\zeta q^2;q^2)_n (-q^3;q^2)_n}\right) = \frac{1}{1+q} \mathcal{R}(\zeta,-q;q^2).
\end{align*}
Hence,
\begin{align*}
{}_2\phi_1(q^2,0;-q;q^2;\zeta) = 1 + \frac{\zeta}{(1-\zeta)(1+q)}\mathcal{R}(\zeta,-q;q^2),
\end{align*}
as claimed.

To verify the second equation, we proceed analogously. Explicitly, we begin with
\begin{align*}
{}_2\phi_1(q^2,0;-q^2;q^2;\zeta) = F(0,-1;\zeta;q^2).
\end{align*}
Then, by equation \eqref{eqn: fine 4.4} we find that
\begin{align*}
F(0,-1; \zeta;q^2) = 1+\frac{\zeta}{1+q^2} F(0,-q^2;\zeta; q^2).
\end{align*}
Next, by \eqref{eqn: fine 6.3}, we obtain
\begin{align*}
F(0,\zeta;-q^2;q^2) = \frac{1-\zeta}{1+q^2}F(0,-q^2;\zeta;q^2).
\end{align*}
Thus,
\begin{align*}
F(0,-1;\zeta;q^2) = 1+\frac{\zeta}{1-\zeta} F(0,\zeta;-q^2;q^2).
\end{align*}
Using \eqref{eqn: fine 12.3} to inspect
\begin{align*}
F(0,\zeta;-q^2;q^2) = \frac{1}{(1+q^2)} \left( \sum_{n \geq 0} \frac{(-q^2\zeta)^n q^{2n^2}}{(\zeta q^2;q^2)_n (-q^4;q^2)_n}\right) = \frac{1}{1+q^2} \mathcal{R}(\zeta,-q^2;q^2).
\end{align*}
This proves the second equation.

For the final equality in the lemma, we see that 
\begin{align*}
{}_2\phi_1(q^2,0;q;q^2;\zeta) = F(0,q^{-1};\zeta;q^2).
\end{align*}
By \eqref{eqn: fine 4.4} this is
\begin{align*}
F(0,q^{-1};\zeta;q^2) = 1+ \frac{\zeta}{1-q} F(0,q;\zeta;q^2).
\end{align*} 
Using \eqref{eqn: fine 6.3}, we then have that
\begin{align*}
F(0,\zeta;q;q^2) = \frac{1-\zeta}{1-q} F(0,q;\zeta;q^2).
\end{align*}
We obtain
\begin{align*}
F(0,q^{-1};\zeta;q^2) = 1+ \frac{\zeta}{1-\zeta}F(0,\zeta;q;q^2).
\end{align*}
Inspecting the final term more closely with \eqref{eqn: fine 12.3}, we find that
\begin{align*}
F(0,\zeta;q;q^2) = \frac{1}{1-q} \left(\sum_{n \geq 0} \frac{(\zeta q)^n q^{2n^2}}{(\zeta q^2;q^2)_n (q^3;q^2)_n} \right) = \frac{1}{1-q} \mathcal{R}(\zeta,q;q^2).
\end{align*}
Combining these yields the third equation, and thus the lemma is proven.
\end{proof}

We now are able to prove our main theorems.

\begin{proof}[Proof of Theorem \ref{Theorem: main}]
We utilize \eqref{eqn: product} to prove the representations of $f_1$, $f_2$, $f_3$ as double-sum $q$-series. We begin with the first case. By Lemma \ref{Lem: mocks to R} and \eqref{eqn: nu as 1phi1} the left-hand side is 
\begin{align*}
{}_1\phi_1(q^2;-q;q^2;-1) {}_2\phi_1(q^2,0;-q;q^2;\zeta).
\end{align*}
 We compute by \eqref{eqn: product} that
\begin{align*}
& {}_1\phi_1(q^2;-q;q^2;-1) {}_2\phi_1(q^2,0;-q;q^2;\zeta) \notag \\
&=  \sum_{m,n\geq0}q^{2n(n-1)}\frac{(q^2;q^2)_{m+n}(-q^{2n}/\zeta;q^2)_m(-q^{-1},q^2)_n}{(-q;q^2)_{m+2n}(-q;q^2)_n}\cdot\frac{\zeta^m}{(q^2;q^2)_m}\cdot\frac{(-q^2\zeta)^n}{(q^2;q^2)_n} \notag\\
&= \sum_{m,n \geq 0} (-1)^n q^{2n^2}\zeta^{n+m} \left[{m+n}\atop{m}\right]_{q^2} \frac{(-q^{2n}/\zeta;q^2)_m(-q^{-1};q^2)_n}{(-q;q^2)_n(-q;q^2)_{m+2n}} \notag\\
&= \sum_{m,n \geq 0} (-1)^n q^{2n^2}\zeta^{n+m} \left[{m+n}\atop{m}\right]_{q^2} \frac{(-q^{2n}/\zeta;q^2)_m (1+q^{-1})}{(1+q^{2n-1})(-q;q^2)_{m+2n}}.
\end{align*}

Next, consider the second case. By Lemma \ref{Lem: mocks to R} and \eqref{eqn: phi as 1phi1}, the left-hand side is
\begin{align*}
{}_1\phi_1(q^2;-q^2;q^2;-q) {}_2\phi_1(q^2,0;-q^2;q^2;\zeta).
\end{align*}
Utilizing \eqref{eqn: product}, we obtain
\begin{align*}
& {}_1\phi_1(q^2;-q^2;q^2;-q) {}_2\phi_1(q^2,0;-q^2;q^2;\zeta) \notag \\
&=  \sum_{m,n\geq0}q^{2n(n-1)}\frac{(q^2;q^2)_{m+n}(-q^{2n+1}/\zeta;q^2)_m(-1;q^2)_n}{(-q^2;q^2)_{m+2n}(-q^2;q^2)_n}\cdot\frac{\zeta^m}{(q^2;q^2)_m}\cdot\frac{(-q^3\zeta)^n}{(q^2;q^2)_n} \notag\\
&= \sum_{m,n \geq 0} (-1)^n q^{2n^2+n}\zeta^{n+m} \left[{m+n}\atop{m}\right]_{q^2} \frac{(-q^{2n+1}/\zeta;q^2)_m(-1;q^2)_n}{(-q^2;q^2)_n(-q^2;q^2)_{m+2n}} \notag\\
&= 2\sum_{m,n \geq 0} (-1)^n q^{2n^2+n}\zeta^{n+m} \left[{m+n}\atop{m}\right]_{q^2} \frac{(-q^{2n+1}/\zeta;q^2)_m}{(1+q^{2n})(-q^2;q^2)_{m+2n}}.
\end{align*}

Finally, we prove the third case. By Lemma \ref{Lem: mocks to R} and \eqref{eqn: rho as 1phi1} the left-hand side is 
\begin{align*}
{}_1\phi_1(q^2;q;q^2;-q) {}_2\phi_1(q^2,0;q;q^2;\zeta).
\end{align*} 
Utilizing \eqref{eqn: product}, we obtain
\begin{align*}
& {}_1\phi_1(q^2;q;q^2;-q) {}_2\phi_1(q^2,0;q;q^2;\zeta) \notag\\
&=  \sum_{m,n\geq0} q^{2n(n-1)}\frac{(q^2;q^2)_{m+n}(-q^{2n+1}/\zeta;q^2)_m(q^{-1};q^2)_n}{(q;q^2)_{m+2n}(q;q^2)_n}\cdot\frac{\zeta^m}{(q^2;q^2)_m}\cdot\frac{(-q^3\zeta)^n}{(q^2;q^2)_n} \notag\\
&= \sum_{m,n \geq 0}(-1)^n q^{2n^2+n}\zeta^{n+m} \left[{m+n}\atop{m}\right]_{q^2} \frac{(-q^{2n+1}/\zeta;q^2)_m(q^{-1};q^2)_n}{(q;q^2)_n(q;q^2)_{m+2n}} \notag\\
&= \sum_{m,n \geq 0} (-1)^n q^{2n^2+n}\zeta^{n+m} \left[{m+n}\atop{m}\right]_{q^2} \frac{(-q^{2n+1}/\zeta;q^2)_m(1-q^{-1})}{(1-q^{2n-1})(q;q^2)_{m+2n}}.
\end{align*}
This proves Theorem \ref{Theorem: main}.
\end{proof}

\begin{proof}[Proof of Theorem \ref{Theorem main: completions}]
Combining Lemmas \ref{Lem: R completion} and \ref{Lem: mock theta completion} immediately yields modular completions of $f_1$, $f_2$, $f_3$ in the obvious fashion. For instance, since
\begin{align*}
f_1(\tau) =  1+\nu(q) + \frac{\zeta}{(1-\zeta)(1+q)}\mathcal{R}(\zeta,-q;q^2) + \nu(q)\frac{\zeta}{(1-\zeta)(1+q)}\mathcal{R}(\zeta,-q;q^2),
\end{align*}
a natural modular completion of $f_1$ is given by
\begin{align*}
-1 -q^{-\frac{1}{2}} R(2\tau;12\tau) + \frac{\zeta}{(1-\zeta)(1+q)}\mathcal{C}(\zeta,-q;q^2) -q^{-\frac{1}{2}} R(2\tau;12\tau)\frac{\zeta}{(1-\zeta)(1+q)}\mathcal{C}(\zeta,-q;q^2).
\end{align*}
The cases of $f_2$ and $f_3$ are completely analogous. This proves that these functions are indeed mock theta functions of depth two.
\end{proof}

\end{document}